\documentclass[11pt, twoside]{article}
\usepackage[centertags]{amsmath}
\usepackage{amsfonts}
\usepackage{amssymb}
\usepackage{amsthm}
\usepackage{newlfont}
\usepackage[latin1]{inputenc}
\usepackage[dvips]{graphics,color}

\newtheorem{fed}{Definition}[section]
\newtheorem{teo}[fed]{Theorem}

\newtheorem{rem}[fed]{Remark}

\newtheorem{ej}[fed]{Example}
\def\ds{\displaystyle}

\begin{document}
\title{A reformulation of an ordinary differential equation
 }
\author{Oscar A. Barraza\footnote{ \textit{E-mail address:}
\texttt{oscar.barraza@econo.unlp.edu.ar }
}\\\scriptsize{Departamento de Ciencias
Complementarias}\\\scriptsize{Facultad de Ciencias Económicas,
UNLP} \footnote{ \textit{Post address:} calle 6 nro. 777, (1900) -
La Plata, Argentina } \footnote{ \textit{Telephone number:}
54-221-4236771, int.128} \footnote{ \textit{Fax number:}
54-221-4236769}}

\date{\vspace{.3cm}January, 2013.}
\maketitle \vskip-2cm


\begin{abstract}

\noindent The purpose of this note is to present a formulation of
a given nonlinear ordinary differential equation into an
equivalent system of linear ordinary differential equations. It is
evident that the easiness of a such procedure would be able to
open a new way in order to calculate or approximate the solution
of an ordinary differential equation. Some examples are presented.

\end{abstract}

\vfill \eject



\section{Introduction}

Let us consider an initial value problem (IVP) of first order like

\begin{equation}\label{PHVI}
\left\{
\begin{array}{rcl}
y'=\dfrac{dy}{dt}\!\! &=&\!\! \varphi(y) \\
y(t_0)\!\! &=&\!\! y_{0},
\end{array}
\right.
\end{equation}
\ \\
 or
\ \\
\begin{equation}\label{PNHVI}
\left\{
\begin{array}{rcl}
y'=\dfrac{dy}{dt}\!\! &=&\!\! \varphi(y) + g(t) \\
y(t_0)\!\! &=&\!\! y_{0},
\end{array}
\right.
\end{equation}
\ \\
where the function $\varphi(y)$ is supposed to be sufficiently
smooth on an open interval $(t_0-\delta, t_0+\delta )$ for some
$\delta
> 0$.

Usually, problems like (\ref{PHVI}) or (\ref{PNHVI}) are solved by
means of known special methods according to the particular type of
function $\varphi$, for example, separation of variables,
homogeneous equations, linear equations, Bernoulli equations, etc.
However, the difficulty to find out an exact solution is related
to some particular facts associated to the particular method of
resolution. For instance, when the method of separation of
variables is attempted to be applied the existence of a general
primitive function of $\dfrac {1}{\varphi(y)}$ is required; but
the fact of computing such a primitive is not a trivial challenge.

The goal in this work consists to transform problems like
(\ref{PHVI}) or (\ref{PNHVI}) into a following system of infinite
linear ordinary differential equations
\begin{equation}\label{SLEDO}
\begin{array}{cc}
  \mathbf{X'=AX}(t),\qquad \text{or} \quad & \mathbf{X}'=\mathbf{AX}(t) + \mathbf{b}(t) \\
\end{array}
\end{equation}
\ \\
such that $\mathbf{X}(t)=\left(%
\begin{array}{c}
  x_0(t) \\
  x_1(t) \\
  x_2(t) \\
  x_3(t) \\
  \vdots \\
\end{array}%
\right)$ in an infinite dimensional function vector, \hfill
 \vskip4mm
\ \\
$\mathbf{A}=\left(%
\begin{array}{ccccc}
  a_{00} & a_{01} & a_{02} & a_{03} & \cdots \\
  a_{10} & a_{11} & a_{12} & a_{13} & \cdots \\
  a_{20} & a_{21} & a_{22} & a_{23} & \cdots \\
  a_{30} & a_{31} & a_{32} & a_{33} & \cdots \\
  \vdots & \vdots & \vdots & \vdots & \ddots \\
\end{array}%
\right)$ is a real constant entries infinite dimensional matrix
and \hfill
 \vskip4mm
\ \\
$\mathbf{b}(t)=\left(%
\begin{array}{c}
  g(t) \\
  0 \\
  0 \\
  0 \\
  \vdots \\
\end{array}%
\right)$ is an other infinite dimensional function vector.
\vskip4mm

The organization of this work is as follows. The aim of Section 2
consists to display some particular transforms through some simple
examples. Section 3 is devoted to expose the procedure to
transform an initial value problem into a system of linear
ordinary differential equations. In this section the equivalence
between the solutions of the IVP and the system is shown. In
Section 4 some suggestions about possible generalizations are
presented.

\section{Some examples}

Before developing the details of the procedure, two examples are
presented in order to show how this technique works.

\begin{ej}

Let us present an easy example because it is an initial value
problem where the ordinary differential equation is linear. That
is,
\begin{equation}\label{1erEjemplo}
\left\{
\begin{array}{rcl}
y'=\dfrac{dy}{dt}\!\! &=&\!\! y-t \\
y(0)\!\! &=&\!\! y_{0},
\end{array}
\right.
\end{equation}
\end{ej}
for some real number $y_0$.

\begin{itemize}
\item It is straightforward to compute the exact solution because
the ODE is linear. Then the general solution is
\[
y(t) = K.e^t + t + 1,
\]
for an arbitrary real constant $K$. Using the initial condition
the constant is $K = y_0 - 1$. So, the solution of the IVP is
\[
y = (y_0 - 1).e^t + t + 1.
\]

\item Let us present the transformation putting
\[
x_0(t) = y(t)
\]
then
\[
x_0'(t) = y'(t)= y(t) - t = x_1(t)\quad \therefore x_0' = x_1.
\]
So,
\[
x_1'(t) = (y(t) - t)' = y'(t) - 1 = y(t) - t - 1 := x_2(t)\quad
\therefore x_1' = x_2.
\]
One more step
\[
x_2'(t) = (y(t) - t - 1)' = y'(t) - 1 = y(t) - t - 1 :=
x_3(t)\quad \therefore x_2' = x_3.
\]
Then, in general for all integer $j\ge 0$ it is deduced that \[
x_j' = x_{j+1}.
\]

In this way the matrix system is
\[
\mathbf{X}'(t) = \underbrace{\left(%
\begin{array}{ccccc}
  0 & 1 & 0 & 0 & \cdots \\
  0 & 0 & 1 & 0 & \cdots \\
  0 & 0 & 0 & 1 & \cdots \\
  0 & 0 & 0 & 0 & \cdots \\
  \vdots & \vdots & \vdots & \vdots & \ddots \\
\end{array}%
\right)}_{= \ A } \mathbf{X}(t) + \underbrace{\left(%
\begin{array}{c}
  -t \\
  0 \\
  0 \\
  0 \\
  \vdots \\
\end{array}%
\right)}_{= \ b(t)}.
\]

It is recalled that the exponential matrix of a matrix $B$
belonging to a Banach algebra $\mathcal{A}$ is defined by
\begin{equation}\label{SERIEEXPONENCIAL}
\begin{array}{cc}
  e^{\mathbf{B}} & = \mathbf{I + B} + \frac{1}{2!}\mathbf{B}^2 +
\frac{1}{3!}\mathbf{B}^3 + \dots = \\
  \ &\ \\
  \ & = \mathbf{I} + \ds \sum_{j=1}^{+\infty} \frac{1}{j!}\mathbf{B}^j = \ds \lim_{n\to +\infty} S_n \\
\end{array}
\end{equation}
\ \\
where $\mathbf{I}$ is the identity matrix and $S_n = \mathbf{I} +
\ds \sum_{j=1}^{n-1} \frac{1}{j!}\mathbf{B}^j$ is the
$n$th-partial sum. This series is absolutely convergent in the
Banach algebra $\mathcal{A}$  because on one side
$$\|\mathbf{B}^k\| \le \|\mathbf{B}\|^k$$ for all integer $k\ge 1,$ and then $$\|e^\mathbf{B}\| \le
1 + \sum_{j=1}^{+\infty} \frac{1}{j!}\|\mathbf{B}\|^j =
e^{\|\mathbf{B}\|},$$
 showing that the exponential matrix $e^B$ is bounded in the norm of the Banach algebra $\mathcal{A}$
 for all matrix $B \in \mathcal{A}$. On the other side the sequence of the partial sums $\left\{ S_n \right\}_n$ is Cauchy
 convergent. For the details the reader is referred, for instance, to (\cite{Ru}).

Taking into account the variation of constants formula, the
solution of the previous linear system is
\[
\mathbf{X}(t) = \ds{e^{A.t}.\mathbf{X}(0) + \int_{0}^t
e^{A(t-s)}.b(s)\, ds }
\]
 and, as a consequence, the solution of the given IPV is then
 \[
 \begin{array}{rl}
  y(t) \!\!\! & = x_0(t) = (1\textrm{st.\ row\ of\ } \mathbf{X}(t)) = \\
& = (1\textrm{st.\ row\ of\ } e^{\mathbf{A}t}).\left(%
\begin{array}{c}
  c_0 \\
  c_1 \\
  c_2 \\
  c_3 \\
  \vdots \\
\end{array}%
\right) + \ds \int_{0}^t (1\textrm{st.\ row\ of\ } e^{\mathbf{A}(t-s)}).\left(%
\begin{array}{c}
  -s \\
  0 \\
  0 \\
  0 \\
  \vdots \\
\end{array}%
\right) \, ds  \\
& = c_0 + c_1 t + \dfrac{1}{2!} c_2 t^2 + \dfrac{1}{3!} c_3 t^3 +
\dots + \ds \int_{0}^t (-s) \, ds = \\
& = c_0 + c_1 t + \dfrac{1}{2!} c_2 t^2 + \dfrac{1}{3!} c_3 t^3 +
\dots \left. -\dfrac{s^2}2 \right|_{0}^t. \\
& = c_0 + c_1 t +
\dfrac{1}{2!} c_2 t^2 + \dfrac{1}{3!} c_3 t^3 +
\dots  -\dfrac{t^2}2. \\
\end{array}
\]
This implies that
 \[
 x_0(0) = c_0 = y_0
 \]
and that from the equations of the system
 \[ x_j(0) = \dfrac{d^j\, x_0}{dt^j}(0) = \left\{%
 \begin{array}{ll}
    c_j, & \hbox{for\  $j\ge 1$\ and\  $j\neq 2$,} \\
    c_2 - 1, & \hbox{ for\ $j = 2$,} \\
 \end{array}%
 \right.
 \]
so, the solution is written as
\[
\begin{array}{ll}
  y(t) \!\!\! & = y_0 + y_0 t + \left( \dfrac{1}{2!} y_0 - \dfrac12 \right)
t^2 + \dfrac{1}{3!}\left( y_0 - 1 \right) t^3 + \dots +
\dfrac{1}{j!}\left( y_0 - 1 \right) t^j + \dots  \\
   & = (y_0 -1).e^t + t + 1. \\
\end{array}
\]

\item The reader can observe that through both procedures the same
solution is reached for all real $t$.

\end{itemize}

\begin{ej}
\begin{equation}\label{2doEjemplo}
\left\{
\begin{array}{rcl}
y'=\dfrac{dy}{dt}\!\! &=&\!\! y^2 \\
y(0)\!\! &=&\!\! y_{0},
\end{array}
\right.
\end{equation}
\end{ej}

\begin{itemize}
\item First of all, the general solution is presented. This
solution can be easily computed by means of the separation of
variables method. The initial value $y_0$ must be different to
zero in order to guarantee the nontrivial solution. Then, under
the assumption $y_0\ne 0$, the exact solution of problem
(\ref{2doEjemplo}) defined for $t\ne\frac1{y_0}$ is given by
\begin{equation}\label{SOLU2doEjemplo}
y(t)=\dfrac{y_0}{1-y_0(t-t_0)}=\dfrac{y_0}{1-y_0.t},
\end{equation} since $t_0 = 0$. It is mandatory to keep the domain
$$\left( -\dfrac1{y_0}, +\infty\right)\quad  \text{if\ }\quad y_0>0$$
or
$$\left( -\infty, \dfrac1{|y_0|}\right)\quad  \text{if\ }\quad y_0<0.$$

For $y_0=0$ the solution obtained is the trivial one $y(t) = 0$,
for all $t$.

\item Now, we show how to transform problem (\ref{2doEjemplo})
into a system of infinite linear ordinary equations.

Let us define
$$x_0(t)=y(t)$$
with $x_0(0)=y_0$. Then,
$$x_1(t):=x_0'(t)=y'(t)=y^2(t),$$ so
$$x_1'(t)=2.y(t).y'(t)=2y^3(t).$$
The next variable $x_2$ is defined by $$x_2(t) =y^3(t),$$ and then
$$x_2'(t)=3.y^2(t).y'(t)=3y^4(t).$$ Continuing with the precedent scheme let
us define $$x_3(t)=y^4(t),$$ and then
$$x_3'(t)=4.y^3(t).y'(t)=4y^5(t),$$ and so on.

The general expression in the present example is $$x_j(t) =
(y(t))^{j+1},$$ for all integer $j\ge 0$.

Thus the set of the infinite linear differential equations can be
summarized by $$x_{j-1}'(t)=j.x_j(t),\qquad \text{for\ all\
integer\ } j\ge 1,$$ with the initial condition $x_0(t_0)=y_0$.
So, this system is expressed in its matrix form as
\begin{equation}\label{2doEjemploMatricial}
\begin{array}{c}
  \mathbf{X}'(t)=\left(%
\begin{array}{c}
  x_1(t) \\
  2.x_2(t) \\
  3.x_3(t) \\
  4.x_4(t) \\
  \vdots \\
\end{array}%
\right) = \underbrace{\left(%
\begin{array}{cccccc}
  0 & 1 & 0 & 0 & 0 & \cdots \\
  0 & 0 & 2 & 0 & 0 & \cdots \\
  0 & 0 & 0 & 3 & 0 & \cdots \\
  0 & 0 & 0 & 0 & 4 & \cdots \\
  \vdots & \vdots & \vdots & \vdots & \vdots & \ddots \\
\end{array}%
\right)}_{=\ \mathbf{A}}\left(%
\begin{array}{c}
  x_0(t) \\
  x_1(t) \\
  x_2(t) \\
  x_3(t) \\
  \vdots \\
\end{array}%
\right) = \\
\ \\
  \hskip-1.7cm = \mathbf{A.X}(t)
 \\
\end{array}
\end{equation}
where $$\mathbf{X}(t)=\left(%
\begin{array}{c}
  x_0(t) \\
  x_1(t) \\
  x_2(t) \\
  x_3(t) \\
  \vdots \\
\end{array}%
\right)$$ and $$\mathbf{A}= \left(%
\begin{array}{cccccc}
  0 & 1 & 0 & 0 & 0 & \cdots \\
  0 & 0 & 2 & 0 & 0 & \cdots \\
  0 & 0 & 0 & 3 & 0 & \cdots \\
  0 & 0 & 0 & 0 & 4 & \cdots \\
  \vdots & \vdots & \vdots & \vdots & \vdots & \ddots \\
\end{array}%
\right).$$

The general solution of system (\ref{2doEjemploMatricial}) has the
form
 \vskip-0.7cm
\begin{equation}\label{SOLULINEAL2}
\mathbf{X}(t)= e^{\mathbf{A}t}.\left(%
\begin{array}{c}
  c_0 \\
  c_1 \\
  c_2 \\
  c_3 \\
  \vdots \\
\end{array}%
\right),
\end{equation}

 \vskip-1.4cm
where $\left(%
\begin{array}{c}
  c_0 \\
  c_1 \\
  c_2 \\
  c_3 \\
  \vdots \\
\end{array}%
\right) $ is the vector of arbitrary constants.
 \vskip-0.1cm

 \item Finally, taking into account (\ref{SOLULINEAL2}), (\ref{SERIEEXPONENCIAL})
and each one of the powers of the matrix $\mathbf{A}$, that is

$\mathbf{A}^2=\left(%
\begin{array}{ccccccc}
  0 & 0 & 2 & 0 & 0 & 0 &\cdots \\
  0 & 0 & 0 & 6 & 0 & 0 &\cdots \\
  0 & 0 & 0 & 0 & 12 & 0 &\cdots \\
  0 & 0 & 0 & 0 & 0 & 20 &\cdots \\
  \vdots & \vdots & \vdots & \vdots & \vdots & \vdots & \ddots \\
\end{array}%
\right),$

$\mathbf{A}^3=\left(%
\begin{array}{ccccccc}
  0 & 0 & 0 & 6 & 0 & 0 & \cdots \\
  0 & 0 & 0 & 0 & 24 & 0 & \cdots \\
  0 & 0 & 0 & 3 & 0 & 60 & \cdots \\
  0 & 0 & 0 & 0 & 0 & 0 & \cdots \\
  \vdots & \vdots & \vdots & \vdots & \vdots & \vdots & \ddots \\
\end{array}%
\right),$

and so on, the general solution of problem (\ref{2doEjemplo}) is
given by
 $$\begin{array}{c}
  y(t) = x_0(t)=\left(\text{1st.\ row\ of\ }e^{\mathbf{A}t}\right).\left(%
 \begin{array}{c}
  c_0 \\
  c_1 \\
  c_2 \\
  c_3 \\
  \vdots \\
 \end{array}%
 \right)=\hskip4.5cm \\
   = \left(\text{1st.\ row\ of\ }\left(\mathbf{I} + \ds \sum_{j=1}^{+\infty}
   \frac1{j!} (\mathbf{A}t)^j\right)\right).\left(%
 \begin{array}{c}
  c_0 \\
  c_1 \\
  c_2 \\
  c_3 \\
  \vdots \\
 \end{array}%
 \right)= \hskip2cm \\
 \end{array}
  $$

$$
\begin{array}{c}
 \hskip1.3cm  = \left(\text{1st.\ row\ of\ }\mathbf{I} + \ds \sum_{j=1}^{+\infty}
   \frac1{j!} \left(\text{1st.\ row\ of\ }(\mathbf{A})^j\right)t^j\right).\left(%
 \begin{array}{c}
  c_0 \\
  c_1 \\
  c_2 \\
  c_3 \\
  \vdots \\
 \end{array}%
 \right)  = \\
 \hskip0.7cm  = c_0 + c_1 t + \frac1{2!}.2! c_2 t^2 + \frac1{3!}.3! c_3 t^3 +
  \dots + \frac1{j!}.j! c_j t^j + \dots =  \\
 \ \\
 \hskip-1.8cm = c_0 + c_1 t + c_2 t^2 + c_3 t^3 + \dots + c_j t^j + \dots  \\ \\
\end{array}
$$

 Since the initial condition satisfies $y_0=y(0)=x_0(0)=c_0$ and
for all integer $j\ge 0$ it is clear that $x_j(t) = (y(t))^{j+1}$,
it is deduced that $c_j = x_j(0) = (y(0))^{j+1} = c_0^{j+1}$ for
all $j \ge 1$. As a consequence the solution calculated above
takes the form
\begin{equation}\label{SOLUserie2}
\begin{array}{c}
 y(t) = c_0 + c_0^2 t + c_0^3 t^2 + c_0^4 t^3 + \dots + c_0^{j+1}
t^j + \dots = \hskip1.9cm \\
\ \\
 = c_0 \left[1 + c_0 t + (c_0 t)^2 + (c_0 t)^3 + \dots
+ (c_0 t)^j + \dots \right] = \\
\ \\
 = c_0.\dfrac{1}{1-c_0.t} =
\dfrac{y_0}{1-y_0.t}, \hskip4.7cm \\
\end{array}
\end{equation}
defined under the assumption $|y_0.t| < 1$. Note that in the open
interval $\left(-\dfrac1{|y_0|}, \dfrac1{|y_0|}\right)$ solutions
(\ref{SOLU2doEjemplo}) and (\ref{SOLUserie2}) agree each to other.
According to the sign of $y_0$, the domain of this solution could
be extended to $+\infty$ or to $-\infty$.

For the special case where $y_0=0$, the solution by the series
expansion is the zero solution as before.
\end{itemize}

\begin{ej}
\begin{equation}\label{3erEjemplo}
\left\{
\begin{array}{rcl}
y'=\dfrac{dy}{dt}\!\! &=&\!\! e^y \\
y(0)\!\! &=&\!\! y_{0},
\end{array}
\right.
\end{equation}
\end{ej}

\begin{itemize}
\item As in the previous example the solution of problem
(\ref{2doEjemplo}) can be rapidly computed by means of the
separation of variables method to obtain for all initial value
$y_0$ that
\begin{equation}\label{SOLU3erEjemplo}
y(t)=-\ln(e^{-y_0} -t),
\end{equation} Its domain, for $y_0$ given, is defined by the condition
$e^{-y_0}-t>0$ which is equivalent to\quad $t < e^{-y_0}$,\quad
that is\quad $e^{y_0}.t<1$.

\item Let us pass to the transformation of problem
(\ref{3erEjemplo}) into a system of infinite linear ordinary
equations.

Let us define
$$x_0(t)=y(t)$$
with $x_0(0)=y_0$. So
$$x_1(t):=x_0'(t)=y'(t)=e^y(t),$$ and
$$x_1'(t)=e^{y(t)}.y'(t)=e^{2y(t)}.$$
Next the variable $x_2$ is defined by $$x_2(t) =e^{2y(t)},$$ and
then
$$x_2'(t)=2.e^{2y(t)}.y'(t)=2e^{3y(t)}.$$ Let
us continue defining $$x_3(t)=e^{3y(t)},$$ thus
$$x_3'(t)=3.e^{3y(t)}.y'(t)=3.e^{4y(t)},$$ and so on.

The general expression in this third example is $$x_j(t) =
e^{j.y(t)},$$ for all integer $j\ge 1$.

Therefore the set of the infinite linear differential equations
can be briefly expressed by $$x_{j-1}'(t)=(j-1).x_j(t),\qquad
\text{for\ all\ integer\ } j\ge 1,$$ with the initial condition
$x_0(t_0)=y_0$. Consequently, this system is expressed in its
matrix form as
\begin{equation}\label{3erEjemploMatricial}
\begin{array}{c}
  \mathbf{X}'(t)=\left(%
\begin{array}{c}
  x_1(t) \\
  x_2(t) \\
  2.x_3(t) \\
  3.x_4(t) \\
  \vdots \\
\end{array}%
\right) = \underbrace{\left(%
\begin{array}{cccccc}
  0 & 1 & 0 & 0 & 0 & \cdots \\
  0 & 0 & 1 & 0 & 0 & \cdots \\
  0 & 0 & 0 & 2 & 0 & \cdots \\
  0 & 0 & 0 & 0 & 3 & \cdots \\
  \vdots & \vdots & \vdots & \vdots & \vdots & \ddots \\
\end{array}%
\right)}_{=\ \mathbf{A}}\left(%
\begin{array}{c}
  x_0(t) \\
  x_1(t) \\
  x_2(t) \\
  x_3(t) \\
  \vdots \\
\end{array}%
\right) = \\
\ \\
  \hskip-1.7cm = \mathbf{A.X}(t)
 \\
\end{array}
\end{equation}
where $$\mathbf{X}(t)=\left(%
\begin{array}{c}
  x_0(t) \\
  x_1(t) \\
  x_2(t) \\
  x_3(t) \\
  \vdots \\
\end{array}%
\right)$$ and $$\mathbf{A}= \left(%
\begin{array}{cccccc}
  0 & 1 & 0 & 0 & 0 & \cdots \\
  0 & 0 & 1 & 0 & 0 & \cdots \\
  0 & 0 & 0 & 2 & 0 & \cdots \\
  0 & 0 & 0 & 0 & 3 & \cdots \\
  \vdots & \vdots & \vdots & \vdots & \vdots & \ddots \\
\end{array}%
\right).$$

The general solution of system (\ref{3erEjemploMatricial}) has the
form
\begin{equation}\label{SOLULINEAL3}
\mathbf{X}(t)= e^{\mathbf{A}t}.\left(%
\begin{array}{c}
  c_0 \\
  c_1 \\
  c_2 \\
  c_3 \\
  \vdots \\
\end{array}%
\right),
\end{equation}
\ \\
where $\left(%
\begin{array}{c}
  c_0 \\
  c_1 \\
  c_2 \\
  c_3 \\
  \vdots \\
\end{array}%
\right) $ is the vector of arbitrary constants.

\item Finally, considering (\ref{SOLULINEAL3}),
(\ref{SERIEEXPONENCIAL}) and each one of the powers of the matrix
$\mathbf{A}$, that is

$$\mathbf{A}^2=\left(%
\begin{array}{ccccccc}
  0 & 0 & 1 & 0 & 0 & 0 &\cdots \\
  0 & 0 & 0 & 2 & 0 & 0 &\cdots \\
  0 & 0 & 0 & 0 & 6 & 0 &\cdots \\
  0 & 0 & 0 & 0 & 0 & 12 &\cdots \\
  \vdots & \vdots & \vdots & \vdots & \vdots & \vdots & \ddots \\
\end{array}%
\right), $$

$$ \mathbf{A}^3=\left(%
\begin{array}{ccccccc}
  0 & 0 & 0 & 2 & 0 & 0 & \cdots \\
  0 & 0 & 0 & 0 & 6 & 0 & \cdots \\
  0 & 0 & 0 & 3 & 0 & 24 & \cdots \\
  0 & 0 & 0 & 0 & 0 & 0 & \cdots \\
  \vdots & \vdots & \vdots & \vdots & \vdots & \vdots & \ddots \\
\end{array}%
\right), $$

and so on, the general solution of problem (\ref{3erEjemplo}) is
given by

 $$\begin{array}{c}
  y(t) = x_0(t)=\left(\text{1st.\ row\ of\ }e^{At}\right).\left(%
 \begin{array}{c}
  c_0 \\
  c_1 \\
  c_2 \\
  c_3 \\
  \vdots \\
 \end{array}%
 \right)=\hskip4.5cm \\
   = \left(\text{1st.\ row\ of\ }\left(\mathbf{I} + \ds \sum_{j=1}^{+\infty} \frac1{j!} (\mathbf{A}t)^j\right)\right).\left(%
 \begin{array}{c}
  c_0 \\
  c_1 \\
  c_2 \\
  c_3 \\
  \vdots \\
 \end{array}%
 \right)= \hskip2cm \\
   = \left(\text{1st.\ row\ of\ }\mathbf{I} + \ds \sum_{j=1}^{+\infty} \frac1{j!} \left(\text{1st.\ row\ of\ }(\mathbf{A})^j\right)t^j\right).\left(%
 \begin{array}{c}
  c_0 \\
  c_1 \\
  c_2 \\
  c_3 \\
  \vdots \\
 \end{array}%
 \right)= \\
  \ \\
 = c_0 + c_1 t + \frac1{2!}. c_2 t^2 + \frac1{3!}.2! c_3 t^3 + \dots + \frac1{j!}.(j-1)! c_j t^j + \dots = \hskip7mm \\
 \ \\
 = c_0 + c_1 t + \frac{c_2}2 t^2 + \frac{c_3}3 t^3 + \dots + \frac{c_j}j t^j + \dots \hskip3.2cm \\
 \end{array}$$

Taking into account that the initial condition is
$y_0=y(0)=x_0(0)=c_0$ and that $x_j(t) = e^{j.y(t)}$ for all
integer $j\ge 1$, it is deduced that $c_j = x_j(0) = e^{j.y(0)} =
e^{j.y_0}$ for all $j \ge 1$. Therefore the solution computed
above takes the form
\begin{equation}\label{SOLUserie3}
\begin{array}{c}
 y(t) = y_0 + e^{y_0} t + \frac1{2}e^{2.y_0} t^2 + \frac1{3}e^{3.y_0} t^3 + \dots + \frac1{j}e^{j.y_0}
t^j + \dots = \hskip1.7cm \\
\ \\
 = y_0 + (e^{y_0}.t) + \frac1{2}(e^{y_0}.t)^2 + \frac1{3}(e^{y_0}.t)^3 + \dots +
 \frac1{j}(e^{y_0}.t)^j + \dots = \\
\ \\
 = y_0-\ln\left(1-e^{y_0}.t\right) =
y_0-\ln\left[e^{y_0}.\left(e^{-y_0}-t\right)\right] = \hskip2.2cm \\
\ \\
= y_0-y_0-\ln\left(e^{-y_0}-t\right) = -
\ln\left(e^{-y_0}-t\right), \hskip3.1cm \\
\end{array}
\end{equation}
\ \\
provided that $|e^{y_0}.t| < 1$. In the open interval
$\left(-e^{-y_0}, e^{-y_0}\right)$ solutions
(\ref{SOLU3erEjemplo}) and (\ref{SOLUserie3}) are equal and the
last one (\ref{SOLUserie3}) can be extended to the interval
$\left(-\infty, e^{-y_0}\right)$, that is \ $e^{y_0}.t < 1$ \  as
in (\ref{SOLU3erEjemplo}).

\end{itemize}


\section{Transforming an initial value problem into a system of linear ordinary differential equations }

The aim of this section is to establish the relationship between
the solution of an initial value problem of kind (\ref{PHVI}) or
(\ref{PNHVI}) and the solution of a system of infinite linear
differential equations of kind (\ref{SLEDO}).

Let us begin with a necessary condition to be satisfied by the
solution $\mathbf{X}(t)$ of a system of infinite linear
differential equations (\ref{SLEDO})
 \[
 \mathbf{X}'=\mathbf{AX}(t) + \mathbf{b}(t) \\
 \]
\ \\
such that $\mathbf{X}(t)=\left(%
\begin{array}{c}
  x_0(t) \\
  x_1(t) \\
  x_2(t) \\
  x_3(t) \\
  \vdots \\
\end{array}%
\right)$ is an infinite dimensional real function vector,
 \vskip2mm
\ \\
$\mathbf{A}=\left(%
\begin{array}{ccccc}
  a_{00} & a_{01} & a_{02} & a_{03} & \cdots \\
  a_{10} & a_{11} & a_{12} & a_{13} & \cdots \\
  a_{20} & a_{21} & a_{22} & a_{23} & \cdots \\
  a_{30} & a_{31} & a_{32} & a_{33} & \cdots \\
  \vdots & \vdots & \vdots & \vdots & \ddots \\
\end{array}%
\right)$ is a real constant entries infinite \linebreak
\ \\
dimensional matrix and $\mathbf{b}(t)=\left(%
\begin{array}{c}
  g(t) \\
  0 \\
  0 \\
  0 \\
  \vdots \\
\end{array}%
\right)$ is an other infinite dimensional real function vector as
it was presented in the introduction. Let us suppose that
$a_{ij}=0$ for $0\le i\le j < +\infty$ and for $j+1< i < +\infty$,
and that $a_{j,j+1}\ne 0$, for all integer $j \ge 0$. In this way
the matrix $\mathbf{A}$ has the explicit form
 \[ \mathbf{A} =
 \left(%
\begin{array}{cccccc}
  0 & a_{01} & 0 & 0 & 0 & \cdots \\
  0 & 0 & a_{12} & 0 & 0 & \cdots \\
  0 & 0 & 0 & a_{23} & 0 & \cdots \\
  0 & 0 & 0 & 0 & a_{34} & \cdots \\
  \vdots & \vdots & \vdots & \vdots & \vdots & \ddots \\
\end{array}%
\right)
 \]
and the equations of the system are clearly given by

\begin{equation}
  \hskip-1.6cm  x_0' = a_{01}\, x_1 + g(t)\\
\end{equation}
\begin{equation}
  \hskip-2.7cm  x_1' = a_{12}\, x_2 \\
\end{equation}
\begin{equation}
 \begin{array}{c}
  \hskip-2.45cm  x_2' = a_{23}\, x_3  \\
   \hskip-2.7cm  \cdots  \\
 \end{array}
\end{equation}
\begin{equation}
 \begin{array}{c}
   x_j' = a_{j,j+1}\, x_{j+1} \quad \text{for\ } j\ge 0 \\
  \hskip-3cm   \cdots  \\
 \end{array}
\end{equation}
\ \\
Let us also assume that the function $\varphi(y)$ is real analytic
at the value $y=y_0$ and the function $g(t)$ is real analytic at
the value $t=t_0$. Therefore, there exists a unique solution
$y(t)$ in some open interval centered at the value $t_0$ for the
IVP (\ref{PNHVI}), which includes the IVP (\ref{PHVI}),
$$
\left\{
\begin{array}{rcl}
y'=\dfrac{dy}{dt}\!\! &=&\!\! \varphi(y) + g(t) \\
y(t_0)\!\! &=&\!\! y_{0},
\end{array}
\right.
$$
Extending the variation of constants formula for the finite
dimensional case ( (\cite{C-L}) or (\cite{H})), it is deduced in a
straightforward way that the solution of the system is
 \begin{equation}\label{INTEGRALSOL}
 \mathbf{X}(t) = e^{\mathbf{A}.(t-t_0)}.\mathbf{C}+
 \int_{t_0}^te^{\mathbf{A}.(t-s)}.\mathbf{b}(s)\, ds
 \end{equation}
where the
infinite dimensional vector of the arbitrary real constants
$\mathbf{C}$ coincides to the vector of the initial values
$\mathbf{X}(t_0)$, that is
\[ \mathbf{C}=
\left(%
\begin{array}{c}
  c_0 \\
  c_1 \\
  c_2 \\
  \vdots \\
\end{array}%
\right) = \mathbf{X}(t_0). \]

\begin{rem}
Minimal hypotheses on smoothness of the vector function operator
$\mathbf{T}$\footnote{It is usually known as the Picard fixed
point theorem for abstract Banach spaces; see, for instance,
(\cite{Ru}). Alternatively, it is called the Contraction Mapping
Principle on complete metric spaces; see, (\cite{M-H}) or
(\cite{Bu}). } on a Banach space of infinite dimensional vector
functions $\mathbf{X}(t)$, where $\mathbf{T}(\mathbf{X})(t) =
\mathbf{A}.\mathbf{X}(t) + \mathbf{b}(t)$ for the case in which
$\mathbf{A}$ is a constant matrix, or $\mathbf{T}(\mathbf{X})(t) =
\mathbf{A}(t)\mathbf{X}(t) + \mathbf{b}(t)$ for the case in which
$\mathbf{A}(t)$ is a function matrix, guarantee a unique local
fixed point of operator $\mathbf{T}$, and then a unique local
solution of the former system of infinite ODE with the initial
vector value $\mathbf{C}$. Besides, as is shown in (\cite{Ru}),
the exponential of an operator on a Banach algebra can be defined
by symbolic calculus, and then it is differentiable satisfying the
standard rule
\[
\dfrac{d\left(e^{\int_{t0}^t \mathbf{A}\, ds}\right)}{dt} =
\mathbf{A}.e^{\int_{t0}^t \mathbf{A}\, ds}.
\]
As a direct consequence, equality (\ref{INTEGRALSOL}) provides the
solution for the linear system (\ref{SLEDO}).
\end{rem}

To reach the goal of this article which is the resolution of
problem (\ref{PNHVI}), the following equality has to be imposed

\begin{equation}\label{y-series}
\begin{array}{rl}
  y(t)\!\!\!\! &= x_0(t)  \\
  \    &= \left( \text{1st\  row\ of\ }
e^{\mathbf{A}.(t-t_0)} \right). \left(%
\begin{array}{c}
  c_0 \\
  c_1 \\
  c_2 \\
  \vdots \\
\end{array}%
\right) + \ds\int_{t_0}^t  \left( \text{1st\  row\ of\ }
e^{\mathbf{A}.(t-s)} \right). \left(%
\begin{array}{c}
  g(s) \\
  0 \\
  0 \\
  \vdots \\
\end{array}%
\right) \, ds  \\
 \ & \ \\
  \  & = c_0 + a_{01}.c_1(t-t_0) + a_{01}.a_{12}.c_2 \dfrac{(t-t_0)^2}{2!} +
  a_{01}.a_{12}.a_{23}.c_3 \dfrac{(t-t_0)^3}{3!} + \dots + \\
  \ & \ \\
  \ &\quad + \left[\ds \prod_{k=0}^{j-1}a_{k,k+1}\right] c_j \dfrac{(t-t_0)^j}{j!} +
  \dots + \ds \int_{t_0}^t g(s)\, ds, \\
\end{array}
\end{equation}

and then by differentiating with respect to $t$ it is easy to
arrive to
 $$
\varphi(y) + g(t) = y'(t) = x_0'(t) = \ds
\sum_{j=1}^{+\infty}\left[\ds \prod_{k=0}^{j-1}a_{k,k+1}\right]
c_j \dfrac{(t-t_0)^{j-1}}{(j-1)!} + g(t),
$$
whence it results that the Taylor series for the function
composition $\varphi(y(t))$ about the value $t = t_0$ is
\begin{equation}\label{phi-series}
 \varphi(y(t)) = \ds
\sum_{h=1}^{+\infty}\left[\ds \prod_{k=0}^{h-1}a_{k,k+1}\right]
c_h \dfrac{(t-t_0)^{h-1}}{(h-1)!}.
\end{equation}

Taking into account the equations (16), (17),(18), \dots, (19) of
the linear system it is evident, for all real $t$ in an open
neighborhood of $t_0$, that
$$
x_0'(t) = a_{01}.x_1(t) + g(t),
$$
so
$$
x_1(t) = \dfrac{x_0'(t) - g(t)}{a_{01}}  =
\dfrac{1}{a_{01}}.\varphi(y(t)) = \dfrac{1}{a_{01}}.\ds
\sum_{h=1}^{+\infty}\left[\ds \prod_{k=0}^{h-1}a_{k,k+1}\right]
c_h \dfrac{(t-t_0)^{h-1}}{(h-1)!},
$$
and for all integer $j\ge 2$ that
$$
x_{j-1}'(t) = a_{j-1,j}.x_{j}(t),
$$
so
$$
\begin{array}{rl}
x_{j}(t) & = \dfrac{1}{a_{j-1,j}}.x_{j-1}'(t) = \dots =
\dfrac{1}{\ds \prod_{k=0}^{j-1}a_{k,k+1}}\ds
\sum_{h=j}^{+\infty}\left[\ds \prod_{k=0}^{h-1}a_{k,k+1}\right]
c_h \dfrac{(t-t_0)^{h-j}}{(h-j)!} =  \\
 \   & = \dfrac{1}{\ds \prod_{k=0}^{j-1}a_{k,k+1}}\ds \dfrac{d^{j}}{dt^j}\Bigl[\varphi(y(t))\Bigr]. \\
\end{array}
$$

And finally, by evaluating at $t=t_0$ it is obtained from
(\ref{y-series}) that
 \vskip2mm

\begin{equation}\label{condinicial}
   \hskip-3.1cm \bullet \ c_0 = y(t_0)= y_0,
   \end{equation}

 \vskip2mm
\ \\
and that
 \vskip2mm
 \begin{equation}\label{coefseries}
  \bullet \ \ds c_j = \dfrac{1}{\ds
 \prod_{k=0}^{j-1}a_{k,k+1}}\left. \ds
  \dfrac{d^{j}}{dt^j}\Bigl[\varphi(y(t))\Bigr]\right|_{t=t_0},
\end{equation}

\vskip2mm
\ \\
for all $j\ge 1$, in coincidence with the construction of the
coefficients of Taylor series (\ref{phi-series}) where $\left.\ds
\dfrac{d^{j}}{dt^j}\Bigl[\varphi(y(t))\Bigr]\right|_{t=t_0} =
\left[ \ds \prod_{k=0}^{j-1}a_{k,k+1}\right] . c_j.$

The previous set of infinite equalities allows to compute the
constants $c_j$ for all integer $j\ge 0$ in terms of the function
$\varphi(y)$, the derivative of $y(t) = x_0(t)$. In fact, the
following statement has been proved.

\begin{teo}
Let us suppose an infinitely derivable function $g(t)$ at $t=t_0$
and the existence of the function $y(t)$ solution of problem
(\ref{PNHVI})
\ \\
\[
\left\{
\begin{array}{rcl}
y'(t)=\dfrac{dy}{dt}\!\! &=&\!\! \varphi(y(t)) + g(t) \\
y(t_0)\!\! &=&\!\! y_{0},
\end{array}
\right.
\]
\ \\
for a function $\varphi(y(t))$ which is represented by its Taylor
series about $t = t_0$
\[
\varphi(y(t)) = \ds \sum_{j\ge 1} \left[\ds \prod_{k=0}^{j-1}
a_{k, k+1}\right] . cj.\dfrac{(t-t_0)^{j-1}}{(j-1)!}
\]
on the open interval $|t - t_0| < \delta $ defined for some real
number $\delta > 0$.

Then, the family of infinite functions given by the following
relations for $ x_0(t) = y(t)  $

 \vskip-4mm

\begin{eqnarray*}
  x_0'(t)\!\!\!\! &=&\!\!\!\! \varphi(y(t)) + g(t) = a_{01}.x_1(t), \\
  x_1'(t)\!\!\!\! &=&\!\!\!\! a_{12}.x_2(t), \\
  x_2'(t)\!\!\!\! &=&\!\!\!\! a_{23}.x_3(t), \\
  \dots\!\!\!\! &\  &\!\!\!\! \dots \\
  x_j'(t)\!\!\!\! &=&\!\!\!\! a_{j,j+1}.x_{j+1}(t), \mathrm{\quad for\ all\ integer\ } j\ge 1
\end{eqnarray*}
\vskip-4mm
\ \\
defines an infinite dimensional function vector

\[\mathbf{X}(t) = \left(\begin{array}{c}
  x_0(t) \\
  x_1(t) \\
  x_2(t) \\
  x_3(t) \\
  \vdots \\
\end{array}
\right)
\]
\ \\
which is the solution of the system of infinite linear ordinary
differential equations (\ref{SLEDO})

\[ \mathbf{X}'(t) = \mathbf{A.X}(t) + \mathbf{b}(t) \]
\ \\
with the initial vector value

\[\mathbf{X}(t_0) = \left(\begin{array}{c}
  c_0 \\
  c_1 \\
  c_2 \\
  c_3 \\
  \vdots \\
\end{array}
\right) . \]
\end{teo}

\begin{proof}
The assumptions about functions $g(t)$ and $\varphi(y(t))$ allows
to establish that the (unique) solution $y(t)$ of the initial
value problem (\ref{PNHVI}) on an open interval centered at the
value $t=t_0$ is infinitely differentiable there thanks to the
application of the chain rule. In this way the existence of the
constants $c_j$, for all integer $j\ge 0$, is guaranteed according
to the family of formulas (\ref{condinicial}) and
(\ref{coefseries}) preceding this theorem. Taking into account
these computations, it is straightforward to conclude that the
function vector $\mathbf{X}(t)$ is the solution of system
(\ref{SLEDO}) satisfying the initial conditions $ \mathbf{X}(t_0)
= \left(\begin{array}{c}
  c_0 \\
  c_1 \\
  c_2 \\
  c_3 \\
  \vdots \\
\end{array}
\right), $ as it was the purpose of this result.

\end{proof}

\begin{teo}
Conversely, let $\mathbf{C} = \left(\begin{array}{c}
  c_0 \\
  c_1 \\
  c_2 \\
  c_3 \\
  \vdots \\
\end{array}
\right) $ be an infinite dimensional real vector and if

\begin{enumerate}
    \item the function $g(t)$ is infinitely differentiable on $t = t_0$,
    \item the function $\varphi(y)$ is infinitely differentiable for all real
        $y$ such that $|y-c_0|< \beta$ for some real number $\beta > 0$,
    \item the power series
     \[
        \ds \sum_{j\ge 1} \left[\ds \prod_{k=0}^{j-1} a_{k, k+1}\right] .
        cj.\dfrac{(t-t_0)^{j-1}}{(j-1)!}
     \]
    converges on the open interval $|t - t_0| < \alpha $ for some
    real number $\alpha > 0$, and
    \item if, besides, system (\ref{SLEDO})
        \[
        \mathbf{X}'(t) = \mathbf{A.X}(t) + \mathbf{b}(t)
        \]
        \ \\
        with initial vector value $\mathbf{X}(t_0) = \mathbf{C} $ admits as solution
        the infinite dimensional function vector
        \[
        \mathbf{X}(t) = \left(\begin{array}{c}
            x_0(t) \\
            x_1(t) \\
            x_2(t) \\
            x_3(t) \\
            \vdots \\
                \end{array}
            \right),
        \]
\end{enumerate}
\ \\
then, its first component $y(t) = x_0(t)$ is the solution of the
initial value problem (\ref{PNHVI}) on an appropriated open
interval $|t-t_0|<\delta $ for some real number $\delta > 0$,
where $y_0 = c_0$.

\end{teo}

\begin{proof}
In order to verify this theorem, let us define $y(t) = x_0(t)$,
the first component of the vector $\mathbf{X}(t)$ solution of
system (\ref{SLEDO})
\[
\mathbf{X}'(t) = \mathbf{A.X}(t) + \mathbf{b}(t)
\]
\ \\
with initial vector value $\mathbf{X}(t_0) =
\left(\begin{array}{c}
  c_0 \\
  c_1 \\
  c_2 \\
  c_3 \\
  \vdots \\
\end{array}
\right) := \mathbf{C}. $

As it was mentioned above, thanks to the variation of constants
method, the solution of this system has the form
\[ \mathbf{X}(t) = e^{\mathbf{A}.(t-t_0)}.\mathbf{C} + \ds \int_{t_0}^t
e^{\mathbf{A}.(t-s)}.b(s)\, ds,\] since $\mathbf{A}$ is a constant
matrix and then it commutes with the exponential one
$e^{\mathbf{A}t}$.
\ \\
Therefore, the first component $y(t) = x_0(t)$ of vector
$\mathbf{X}(t)$, is given by the first row rule; that is,
\[
\begin{array}{rl}
  y(t)\!\!\!\! &= x_0(t)  \\
  \    &= \left( \text{1st\  row\ of\ }
e^{\mathbf{A}.(t-t_0)} \right). \left(%
\begin{array}{c}
  c_0 \\
  c_1 \\
  c_2 \\
  \vdots \\
\end{array}%
\right) + \ds\int_{t_0}^t  \left( \text{1st\  row\ of\ }
e^{\mathbf{A}.(t-s)} \right). \left(%
\begin{array}{c}
  g(s) \\
  0 \\
  0 \\
  \vdots \\
\end{array}%
\right) \, ds  \\
 \ & \ \\
  \  & = c_0 + a_{01}.c_1(t-t_0) + a_{01}.a_{12}.c_2 \dfrac{(t-t_0)^2}{2!} +
  a_{01}.a_{12}.a_{23}.c_3 \dfrac{(t-t_0)^3}{3!} + \dots + \\
  \ & \ \\
  \ &\quad + \left[\ds \prod_{k=0}^{j-1}a_{k,k+1}\right] c_j \dfrac{(t-t_0)^j}{j!}
  + \dots + \ds \int_{t_0}^t g(s)\, ds. \\
\end{array}
\]
\ \\
Or more briefly
\[
y(t) = c_0 + \ds \sum_{j=1}^{+\infty}\ds
\left[\prod_{k=0}^{j-1}a_{k,k+1}\right] c_j \dfrac{(t-t_0)^j}{j!}
  + \ds \int_{t_0}^t g(s)\, ds,
  \]
\ \\
which has sense for all real $t$ in the interval of convergence
$|t-t_0|< \alpha $ of the last series, for some real number
$\alpha
> 0$.

First, by evaluating on $t_0$ it is clear that $y(t_0) = c_0 =
y_0$.

Next, by differentiating term to term the series with respect to
the variable $t$ in the convergence interval $|t-t_0|< \alpha$ and
applying the fundamental calculus theorem it is evident that
\[
\begin{array}{rl}
y'(t) &= \ds \sum_{j=1}^{+\infty}\ds
\left[\prod_{k=0}^{j-1}a_{k,k+1}\right] c_j
\dfrac{(t-t_0)^{j-1}}{(j-1)!} + g(t). \\
\end{array}
\]

By hypothesis, function $g(t)$ is infinitely differentiable on $t
= t_0$ and also the last power series on the open interval $|t -
t_0|< \alpha $. Thus, the function $y(t)$ is infinitely
differentiable on $t=t_0$ as well, so by applying the chain rule
indefinitely to the function composition $\varphi(y(t))$ on $t_0$
the identities (\ref{coefseries}) are rapidly recovered. As a
consequence the series
\[
\ds \sum_{j=1}^{+\infty}\ds
\left[\prod_{k=0}^{j-1}a_{k,k+1}\right] c_j
\dfrac{(t-t_0)^{j-1}}{(j-1)!}
\]
becomes the Taylor series of the function $\varphi(y(t))$ on an
appropriated interval $|t-t_0|<\delta $, for some real number
$\delta$ depending on $t_0$, on $\alpha$, on $\beta$ and on the
values of $g(t)$ on a neighborhood of $t_0$. This fact shows that
$y'(t) = \varphi(y(t)) + g(t)$ proving that $y(t)$ is the solution
of the initial value problem (\ref{PNHVI}), as desired.

\end{proof}

\begin{rem}
According to a particular situation, the domain of the definition
of the solution concerning the power series could be extended to a
larger interval. In fact, this was observed in the two previous
examples.
\end{rem}

\subsection{ An special extension of the method}

Let us now continue with the IVP
 \[
\left\{
\begin{array}{rcl}
y'(t)=\dfrac{dy}{dt}\!\! &=&\!\! \Phi(y(t), t) = \varphi(y(t)).f(t) + g(t) \\
y(t_0)\!\! &=&\!\! y_{0},
\end{array}
\right.
\]
 which its corresponding system of infinite linear
differential equations is supposed to be
\begin{equation}\label{SNLEDO}
 \mathbf{X}'=\mathbf{A}(t).\mathbf{X}(t) + \mathbf{b}(t) \\
\end{equation}
\ \\
such that as before $\mathbf{X}(t)=\left(%
\begin{array}{c}
  x_0(t) \\
  x_1(t) \\
  x_2(t) \\
  x_3(t) \\
  \vdots \\
\end{array}%
\right)$ is an infinite dimensional real function vector,
 \vskip2mm

\begin{equation}\label{MatrixAdet}
\mathbf{A}(t)=\left(%
\begin{array}{ccccc}
  0 & a_{01}f(t) & 0 & 0 & \cdots \\
  0 & 0 & a_{12}f(t) & 0 & \cdots \\
  0 & 0 & 0 & a_{23}f(t) & \cdots \\
  0 & 0 & 0 & 0 & \cdots \\
  \vdots & \vdots & \vdots & \vdots & \ddots \\
\end{array}%
\right) = A.f(t)
\end{equation}
 \vskip2mm
\ \\
is an infinite dimensional matrix with real function entries,
where the matrix
 \vskip2mm
 \ \\
$\mathbf{A} = \left(%
\begin{array}{ccccc}
  0 & a_{01} & 0 & 0 & \cdots \\
  0 & 0 & a_{12} & 0 & \cdots \\
  0 & 0 & 0 & a_{23} & \cdots \\
  0 & 0 & 0 & 0 & \cdots \\
  \vdots & \vdots & \vdots & \vdots & \ddots \\
\end{array}%
\right)$ has constant entries and
$\mathbf{b}(t)=\left(%
\begin{array}{c}
  g(t) \\
  0 \\
  0 \\
  0 \\
  \vdots \\
\end{array}%
\right)$
 \vskip2mm
 \ \\
is an other infinite dimensional real function vector. Thus, the
system is explicitly given by

\begin{equation}
  \hskip-1.6cm  x_0' = a_{01}.f(t)\, x_1 + g(t)\\
\end{equation}
\begin{equation}
  \hskip-2.7cm  x_1' = a_{12}.f(t)\, x_2 \\
\end{equation}
\begin{equation}
 \begin{array}{c}
  \hskip-2.7cm  x_2' = a_{23}.f(t)\, x_3  \\
   \hskip-3cm  \cdots  \\
 \end{array}
\end{equation}
\begin{equation}
 \begin{array}{c}
   x_j' = a_{j,j+1}.f(t)\, x_{j+1} \quad \text{for\ } j\ge 1 \\
  \hskip-3cm   \cdots  \\
 \end{array}
\end{equation}
\ \\ Since the matrix $\mathbf{A}(t) = \mathbf{A}.f(t)$ and the
integral  $\ds \int_{t_0}^t \mathbf{A}(s)\, ds = \mathbf{A}.\ds
\int_{t_0}^t f(s)\, ds$ commute, then $A(t)$ and the exponential
matrix $\ds {e^{ \int_{t_0}^t\mathbf{A}(s)\, ds}}$
$=e^{\mathbf{A}. \int_{t_0}^t f(s)\, ds}$ commute as well, and
thanks to the variation of constants formula the solution of the
previous linear system can be expressed as
\begin{equation}\label{SOLGRAL}
\begin{array}{cc}
\mathbf{X}(t) & = \ds {e^{ \int_{t_0}^t \mathbf{A}(s)\,
ds}.\mathbf{X}(t_0) + \ds \int_{t_0}^t {e^{ \int_r^t
\mathbf{A}(s)\, ds} .\mathbf{b}(r)\, dr}}= \\
 & = \ds {e^{\mathbf{A}. \int_{t_0}^t f(s)\, ds}.\mathbf{X}(t_0) +
\ds \int_{t_0}^t {e^{\mathbf{A}. \int_r^t f(s) \, ds}.\mathbf{b}(r)\, dr} }. \\
\end{array}
\end{equation}
 See, for instance, (\cite{C-L}) or (\cite{H}).

Thus, under little adaptations to this case, it is not difficult
to deduced the corresponding extensions of the previous two
theorems.

\begin{teo}
Let us suppose an infinitely derivable function $g(t)$ at $t=t_0$,
the existence of the function $y(t)$ solution of problem
\ \\
\begin{equation}\label{ExtPNHVI}
\left\{
\begin{array}{rcl}
y'(t)=\dfrac{dy}{dt}\!\! &=&\!\! \Phi(y(t), t) = \varphi(y(t)).f(t) + g(t) \\
y(t_0)\!\! &=&\!\! y_{0},
\end{array}
\right.
\end{equation}
\ \\
for a function $\varphi(y(t))$which is represented by its Taylor
series about $t = t_0$
\[
\varphi(y(t)) = \ds \sum_{j\ge 1} \left[\ds \prod_{k=0}^{j-1}
a_{k, k+1}\right] . cj.\dfrac{(t-t_0)^{j-1}}{(j-1)!}
\]
on an open interval $|t - t_0| < \delta $ defined for some real
number $\delta > 0$, and an integrable function $f(t)$ on the same
interval.

Then, the family of infinite functions defined by the following
relations

\[ x_0(t)=y(t)\]

\[ x_0'(t) = \varphi(y(t)).f(t) + g(t) = a_{01}.f(t)\, x_1(t), \]

\[ x_1'(t) = a_{12}.f(t)\, x_2(t), \]

\[ x_2'(t) = a_{23}.f(t)\, x_3(t), \]

\[ \dots \qquad \dots \]

\[ x_j'(t) = a_{j,j+1}.f(t)\, x_{j+1}(t), \mathrm{\quad for\ all\ integer\ } j\ge 1
\]

\vskip2mm
\ \\
defines an infinite dimensional function vector

\[\mathbf{X}(t) = \left(\begin{array}{c}
  x_0(t) \\
  x_1(t) \\
  x_2(t) \\
  x_3(t) \\
  \vdots \\
\end{array}
\right)
\]
\ \\
which is the local solution of the system of infinite linear
ordinary differential equations (\ref{SNLEDO})

\[ \mathbf{X}'(t) = \mathbf{A}(t).\mathbf{X}(t) + \mathbf{b}(t) \]
\ \\
with the initial vector value

\[\mathbf{X}(t_0) = \left(\begin{array}{c}
  c_0 \\
  c_1 \\
  c_2 \\
  c_3 \\
  \vdots \\
\end{array}
\right) , \]
\end{teo}
and where $\mathbf{A}(t)$ is given by (\ref{MatrixAdet}).

\begin{teo}
Conversely, let $\mathbf{C} = \left(\begin{array}{c}
  c_0 \\
  c_1 \\
  c_2 \\
  c_3 \\
  \vdots \\
\end{array}
\right) $ be an infinite dimensional real vector and if

\begin{enumerate}
    \item the function $g(t)$ is infinitely differentiable on $t = t_0$,
    \item the function $\varphi(y)$ is infinitely differentiable for all real
        $y$ such that $|y-c_0|< \beta$ for some real number $\beta > 0$,
    \item the power series
     \[
        \ds \sum_{j\ge 1} \left[\ds \prod_{k=0}^{j-1} a_{k, k+1}\right] .
        cj.\dfrac{(t-t_0)^{j-1}}{(j-1)!}
     \]
    converges on the open interval $|t - t_0| < \alpha $ for some
    real number $\alpha > 0$, and
    \item if, besides, system (\ref{SLEDO})
        \[
        \mathbf{X}'(t) = \mathbf{A.X}(t) + \mathbf{b}(t)
        \]
        \ \\
        with initial vector value $\mathbf{X}(t_0) = \mathbf{C} $ admits as solution
        the infinite dimensional function vector
        \[
        \mathbf{X}(t) = \left(\begin{array}{c}
            x_0(t) \\
            x_1(t) \\
            x_2(t) \\
            x_3(t) \\
            \vdots \\
                \end{array}
            \right),
        \]
\end{enumerate}
\ \\
then, its first component $y(t) = x_0(t)$ is the solution of the
initial value problem (\ref{PNHVI}) on an appropriated open
interval $|t-t_0|<\delta $ for some real number $\delta > 0$,
where $y_0 = c_0$.

\end{teo}

In order to illustrate this extension, let us present the
following example.

\begin{ej}
\begin{equation}\label{4toEjemplo}
\left\{
\begin{array}{rcl}
y'=\dfrac{dy}{dt}\!\! &=&\!\! y^2.t \\
y(0)\!\! &=&\!\! y_{0},
\end{array}
\right.
\end{equation}
\end{ej}

\begin{itemize}
\item The general solution is easily computed by means of the
separation of variables method. If the initial value $y_0 \neq 0$
 the solution is a nontrivial one; then, the exact solution of problem (\ref{4toEjemplo})
 is given by
\begin{equation}\label{SOLU4toEjemplo}
y(t)=\dfrac{2y_0}{2-y_0.t^2}.
\end{equation}
Its domain is given by the condition
$$y_0.t^2 \neq 2$$

For $y_0=0$ the solution obtained is the trivial one $y(t) = 0$,
for all $t$, which can be included in the previous formula.

\item Let us pass to the transformed of problem (\ref{4toEjemplo})
into a system of infinite linear ordinary equations.

Let us define
$$x_0(t)=y(t)$$
with $y_0 = x_0(0) = c_0$. Then,
$$x_0'(t)=y'(t)=y^2(t).t = t.x_1(t),$$
where $$x_1(t)= y^2(t).$$
So
$$x_1'(t)=2.y(t).y'(t) = 2t.y^3(t) = 2t.x_2(t),$$
where $$x_2(t)= y^3(t).$$
Then
$$x_2'(t)=3.y^2(t).y'(t) = 3t.y^4(t) = 3t.x_3(t),$$ and so on.
In this way the general expression is
$$x_j(t) = y^{j+1}(t) = x_0^{j+1}(t),$$ and
$$x_j'(t)=(j+1).y^j(t).y'(t)=(j+1).t.y^{j+2}(t) = (j+1).t.x_{j+1}(t),$$
for all integer $j\ge 0$.

Thus, the system is expressed in its matrix form as
\begin{equation}\label{4toEjemploMatricial}
\begin{array}{c}
  \mathbf{X}'(t)=\left(%
\begin{array}{c}
  t.x_1(t) \\
  2t.x_2(t) \\
  3t.x_3(t) \\
  4t.x_4(t) \\
  \vdots \\
\end{array}%
\right) = \underbrace{\left(%
\begin{array}{cccccc}
  0 & t & 0 & 0 & 0 & \cdots \\
  0 & 0 & 2t & 0 & 0 & \cdots \\
  0 & 0 & 0 & 3t & 0 & \cdots \\
  0 & 0 & 0 & 0 & 4t & \cdots \\
  \vdots & \vdots & \vdots & \vdots & \vdots & \ddots \\
\end{array}%
\right)}_{=\ \mathbf{A}(t) =\ \mathbf{A}.t}\left(%
\begin{array}{c}
  x_0(t) \\
  x_1(t) \\
  x_2(t) \\
  x_3(t) \\
  \vdots \\
\end{array}%
\right) = \\
\ \\
  \hskip-1.7cm = \mathbf{A}(t).\mathbf{X}(t)
 \\
\end{array}
\end{equation}
where $$\mathbf{X}(t)=\left(%
\begin{array}{c}
  x_0(t) \\
  x_1(t) \\
  x_2(t) \\
  x_3(t) \\
  \vdots \\
\end{array}%
\right)$$ and $$\mathbf{A}= \left(%
\begin{array}{cccccc}
  0 & 1 & 0 & 0 & 0 & \cdots \\
  0 & 0 & 2 & 0 & 0 & \cdots \\
  0 & 0 & 0 & 3 & 0 & \cdots \\
  0 & 0 & 0 & 0 & 4 & \cdots \\
  \vdots & \vdots & \vdots & \vdots & \vdots & \ddots \\
\end{array}%
\right).$$

The general solution of system (\ref{4toEjemploMatricial}) has the
form
\begin{equation}\label{SOLULINEAL4}
\mathbf{X}(t)= {\text{\Large e}}^{ \int_0^t\mathbf{A}.s\, ds}.\left(%
\begin{array}{c}
  c_0 \\
  c_1 \\
  c_2 \\
  c_3 \\
  \vdots \\
\end{array}%
\right),
\end{equation}
\ \\
where $\left(%
\begin{array}{c}
  c_0 \\
  c_1 \\
  c_2 \\
  c_3 \\
  \vdots \\
\end{array}%
\right) $ is the vector of arbitrary constants.

\vskip4mm

\item Finally, taking into account (\ref{SOLULINEAL4}),
(\ref{SERIEEXPONENCIAL}), the matrix $\ds \int_0^t\mathbf{A}.s\,
ds = \mathbf{A}.\ds \int_0^t s\, ds= \mathbf{A}.\dfrac{t^2}{2}$
and each one of the powers of the matrix $\mathbf{A}$, that is
 $$\mathbf{A}^2=\left(%
\begin{array}{ccccccc}
  0 & 0 & 2 & 0 & 0 & 0 &\cdots \\
  0 & 0 & 0 & 6 & 0 & 0 &\cdots \\
  0 & 0 & 0 & 0 & 12 & 0 &\cdots \\
  0 & 0 & 0 & 0 & 0 & 20 &\cdots \\
  \vdots & \vdots & \vdots & \vdots & \vdots & \vdots & \ddots \\
\end{array}%
\right),$$ $$\mathbf{A}^3=\left(%
\begin{array}{ccccccc}
  0 & 0 & 0 & 6 & 0 & 0 & \cdots \\
  0 & 0 & 0 & 0 & 24 & 0 & \cdots \\
  0 & 0 & 0 & 3 & 0 & 60 & \cdots \\
  0 & 0 & 0 & 0 & 0 & 0 & \cdots \\
  \vdots & \vdots & \vdots & \vdots & \vdots & \vdots & \ddots \\
\end{array}%
\right),$$ and so on, the general solution of problem
(\ref{4toEjemplo}) is given by
 $$\begin{array}{c}
  y(t) = x_0(t)=\left(\text{1st.\ row\ of\ }\left(e^{\mathbf{A}.\dfrac{t^2}2}\right)\right).\left(%
 \begin{array}{c}
  c_0 \\
  c_1 \\
  c_2 \\
  c_3 \\
  \vdots \\
 \end{array}%
 \right)=\hskip4.5cm \\
   = \left(\text{1st.\ row\ of\ }\left(\mathbf{I} + \ds \sum_{j=1}^{+\infty} \frac1{j!}
   \left(\mathbf{A}.\dfrac{t^2}2\right)^j\right)\right).\left(%
 \begin{array}{c}
  c_0 \\
  c_1 \\
  c_2 \\
  c_3 \\
  \vdots \\
 \end{array}%
 \right)= \hskip2cm \\
   = \left(\text{1st.\ row\ of\ }\mathbf{I} + \ds \sum_{j=1}^{+\infty} \frac1{j!}
   \left(\text{1st.\ row\ of\ }(\mathbf{A})^j.\left(\dfrac{t^2}{2}\right)^j\right)\right).\left(%
 \begin{array}{c}
  c_0 \\
  c_1 \\
  c_2 \\
  c_3 \\
  \vdots \\
 \end{array}%
 \right)= \\
  \ \\
 = c_0 + c_1 .\dfrac{t^2}2 + \frac1{2!}.2! c_2 .\left(\dfrac{t^2}2\right)^2 +
 \frac1{3!}.3! c_3 \left(\dfrac{t^2}2\right)^3 + \dots + \frac1{j!}.j! c_j \left(\dfrac{t^2}2\right)^j + \dots =
 \hskip7mm \\
 \ \\
 = c_0 + c_1 \left(\dfrac{t^2}2\right) + c_2 \left(\dfrac{t^2}2\right)^2 + c_3 \left(\dfrac{t^2}2\right)^3
 + \dots + c_j \left(\dfrac{t^2}2\right)^j + \dots \hskip3.2cm \\
 \end{array}$$

 Since the initial condition satisfies $y_0=y(0)=x_0(0)=c_0$ and
for all integer $j\ge 0$ it is clear that $x_j(t) = (y(t))^{j+1}$,
it is deduced that $c_j = x_j(0) = (y(0))^{j+1} = c_0^{j+1}$ for
all $j \ge 1$. As a consequence the solution calculated above
takes the form
\begin{equation}\label{SOLUserie4}
\begin{array}{c}
 y(t) = y_0 + y_0^2 \left(\dfrac{t^2}2\right) + y_0^3 \left(\dfrac{t^2}2\right)^2 + y_0^4 \left(\dfrac{t^2}2\right)^3
 + \dots + y_0^{j+1} \left(\dfrac{t^2}2\right)^j + \dots = \hskip1.9cm \\
\ \\
 = y_0 \left[1 + y_0 \left(\dfrac{t^2}2\right) + \left(y_0 \dfrac{t^2}2\right)^2 + \left(y_0 \dfrac{t^2}2\right)^3
 + \dots + \left(y_0 \dfrac{t^2}2\right)^j + \dots \right] = \\
\ \\
 = y_0.\dfrac{1}{1-\left(\dfrac{y_0 t^2}2\right)} =
\dfrac{2y_0}{2-y_0.t^2}, \hskip4.7cm \\
\end{array}
\end{equation}
defined under the assumption $|y_0.t^2| < 2$, that is  $t^2 <
\dfrac{2}{|y_0|}$, if $y_0\neq 0$. Note that in the open interval
$-\sqrt{\dfrac2{|y_0|}} < t < \sqrt{\dfrac2{|y_0|}},$ solutions
(\ref{SOLU4toEjemplo}) and (\ref{SOLUserie4}) agree each to other
and solution (\ref{SOLUserie4}) accepts an extension to all real
number $t\neq \pm \sqrt{\dfrac2{|y_0|}}$

Besides, for $y_0 = 0$ the solution (\ref{SOLUserie4}) is the zero
series.

\end{itemize}


\section{Some conclusions}
\ \\ \vskip-1cm

It is impossible to deny the importance of the role of the
differential equations in all the processes which involve a
dynamical component of a model for continuous variables. These
processes are illustrated in a large variety of domains related,
for instance, to applications to Astronomy, Biology, Chemistry,
Economy, Engineering and Physics.

The necessity to describe the behavior of the exact or
approximated solutions of those dynamical models provides an
understanding of the underlying processes. The fact of computing
the exact solutions of a differential equation is a desirable
purpose but in some cases it is a hard task to accomplish, in most
of them it is impossible. So, finding alternative ways to solve
ordinary differential equations are always good news or, at least,
it could open some paths to obtain approximated solutions. For
example, keeping in mind that the solution of the system
\[
\mathbf{X}'(t) = \mathbf{A}(t).\mathbf{X}(t) + \mathbf{b}(t)
\]
with initial vector value \quad $X(t_0) = \left(%
\begin{array}{c}
  c_0 \\
  c_1 \\
  c_2 \\
  c_3 \\
  \vdots \\
\end{array}%
\right) = \mathbf{C}$ \quad is given by formula (\ref{SOLGRAL})
\[
\mathbf{X}(t) = e^{\mathbf{M}(t)}. \mathbf{X}(t_0) + \ds
\int_{t_0}^t e^{[\mathbf{M}(t)-\mathbf{M}(r)]}.\mathbf{b}(r)\, dr,
\]
where the matrix $\mathbf{M}(t) = \ds \int_{t_0}^t \mathbf{A}(s)\,
ds$, one possibility consists in replacing the exponential matrix
in that formula (\ref{SOLGRAL}) by its $n$-th partial sum. In this
way the sequence of approximations to the desired solution is
defined by
\[
\mathbf{X}_0(t) = \mathbf{X}(t_0) = \mathbf{C}
\]
and for each integer $n\ge 1$
\[
\mathbf{X}_n(t) = \left( \mathbf{I} + \ds
\sum_{j=1}^{n-1}\dfrac{[\mathbf{M}(t)]^j}{j!}\right)\mathbf{X}(t_0)
+ \ds \int_{t_0}^t \left( \mathbf{I} + \ds
\sum_{j=1}^{n-1}\dfrac{[\mathbf{M}(t)-\mathbf{M}(s)]^j}{j!}\right)
\mathbf{b}(s)\, ds.
\]
For the special case of a constant infinite dimensional matrix
$\mathbf{A}(t) = \mathbf{A}$, the matrix $\mathbf{M}(t)$ takes the
form $\mathbf{M}(t) = \ds \int_{t_0}^t \mathbf{A}\, ds = (t-t_0).
\mathbf{A}$ and then the last sequence is reduced to
\[
\mathbf{X}_n(t) = \left( \mathbf{I} + \ds \sum_{j=1}^{n-1}
\dfrac{(t-t_0)^j}{j!}\mathbf{A}^j\right)\mathbf{X}(t_0) + \ds
\int_{t_0}^t \left( \mathbf{I} + \ds
\sum_{j=1}^{n-1}\dfrac{(t-s)^j}{j!}
\mathbf{A}^j\right)\mathbf{b}(s)\, ds,
\]
for each integer $n\ge 1$.

In particular, the method presented in this note which consists to
transform a nonlinear differential equation into a system of
infinite linear ordinary differential equations has the advantage
to translate the difficulty to treat directly with nonlinearity
into the problem of the treatment of the infinite quantity of
linear ODE with constant coefficients or simple variable ones. The
methodology of reducing the difficult of nonlinearity keeps a
certain analogy between the methodology of reducing a $n$th-order
ODE to a first-order system of $n$ ODE because, in general,
solving an isolated linear ODE with constant coefficients is
easier than solving a nonlinear one.

However, the transformation of a nonlinear ODE to a system of
infinite linear ODE's is only developed here for the case of
first-order differential equations; this is the first restriction
of our method. A second limitation is due to the hypotheses about
the smoothness of functions $\varphi$ and $g$ in problems
(\ref{PHVI}) and (\ref{PNHVI}) and the existence of the series
expansion detailed in the theorem. The third constraint of this
method is the form of the derivative $y'(t)$ in the first two
problems. A more general situation to be studied with this
procedure in a next work could concern an ordinary differential
equation like $y'(t) = \varphi(y,t)$, for more general function
$\varphi(y,t)$.

\vskip4mm

\noindent {\Large \textbf{Acknoledgements}}

\vskip4mm

The author wants to thanks Professor Andr\'es Kowalski,
Departamento de Física, Facultad de Ciencias Exactas, UNLP, who
interested him to study the present subject.

\end{document}